\documentclass[12pt, reqno]{amsart}

\author[A.~Aveni]{Andrea Aveni}
\address{Department of Statistical Science,
Duke University, 214 Old Chemistry, Durham, NC 27708-0251}
\email{andrea.aveni@duke.edu}

\author[P.~Leonetti]{Paolo Leonetti}
\address{Department of Economics, Universit\`a degli Studi dell'Insubria, via Monte Generoso 71, Varese 21100, Italy} 
\email{leonetti.paolo@gmail.com}
\urladdr{\url{https://sites.google.com/site/leonettipaolo/}}


\title{Most numbers are not normal}

\usepackage[T1]{fontenc}
\usepackage{amsmath}
\usepackage{amssymb}
\usepackage{amsthm}
\usepackage[left=2.5cm, right=2.5cm, top=3cm]{geometry}
\usepackage{hyperref}
\usepackage{fancyhdr}
\usepackage[inline]{enumitem}
\usepackage{comment}
\usepackage{nicefrac}
\usepackage{bm}
\usepackage{mathrsfs}
\usepackage{graphicx}
\usepackage[utf8]{inputenc}
\usepackage{cancel}
\usepackage{mathtools}

\newcommand{\vertiii}[1]{{\left\vert\kern-0.25ex\left\vert\kern-0.25ex\left\vert #1 
    \right\vert\kern-0.25ex\right\vert\kern-0.25ex\right\vert}}
		
\AtBeginDocument{%
   \def\MR#1{}
}

\newtheorem{thm}{Theorem}[section]
\newtheorem{cor}[thm]{Corollary}

\newtheorem{prop}[thm]{Proposition}

\theoremstyle{definition} 
\newtheorem{defi}[thm]{Definition}
\let\olddefi\defi
\renewcommand{\defi}{\olddefi\normalfont}
\newtheorem{example}[thm]{Example}
\let\oldexample\example
\renewcommand{\example}{\oldexample\normalfont}
\newtheorem{rmk}[thm]{Remark}
\let\oldrmk\rmk
\renewcommand{\rmk}{\oldrmk\normalfont}

\theoremstyle{remark}

\newtheorem*{claim*}{\textsc{Claim}}

\hypersetup{
    pdftitle={Most number are not normal},
    pdfauthor={Paolo Leonetti and Andrea Aveni},
    pdfmenubar=false,
    pdffitwindow=true,
    pdfstartview=FitH,
    colorlinks=true,
    linkcolor=blue,
    citecolor=green,
    urlcolor=cyan
}

\providecommand{\MR}[1]{}

\providecommand{\MR}{\relax\ifhmode\unskip\space\fi MR }

\providecommand{\href}[2]{#2}

\begin{document}

\begin{abstract}
\noindent 
We show, from a topological viewpoint, that most numbers are not normal in a strong sense. More precisely, the set of numbers $x \in (0,1]$ with the following property is comeager: 
for all integers $b\ge 2$ and $k\ge 1$, the sequence of vectors made by the frequencies of all possibile strings of length $k$ in the $b$-adic representation of $x$ has a maximal subset of accumulation points, and each of them is the limit of a subsequence with an index set of nonzero asymptotic density. 
This extends and provides a streamlined proof of 
the main result given by Olsen in [Math. Proc. Cambridge Philos. Soc. \textbf{137} (2004), 43--53]. 
We provide analogues in the context of analytic P-ideals and regular matrices. 
\end{abstract}

\maketitle
\thispagestyle{empty}

\section{Introduction}

A real number $x \in (0,1]$ is normal if, informally, for each base $b\ge 2$, its $b$-adic expansion contains every finite string with the expected uniform limit frequency (the precise definition is given in the next few lines). 
It is well known that most numbers $x$ are normal from a measure theoretic viewpoint, see e.g. \cite{Berg2020} for history and generalizations. 
However, it has been recently shown that certain subsets of nonnormal numbers 
may have full Hausdorff dimension, see e.g. \cite{MR2166730, MR1759398}. 
The aim of this work is to show that, from a topological viewpoint,  most numbers are not normal in a strong sense. 
This provides another nonanalogue between measure and category, cf. \cite{MR584443}. 

For each $x \in (0,1]$, denote its unique nonterminating $b$-adic expansion by 
\begin{equation}\label{eq:badic}
x=\sum\nolimits_{n\ge 1}\frac{d_{b,n}(x)}{b^n},
\end{equation}
with each digit $d_{b,n}(x) \in \{0,1,\ldots,b-1\}$, where $b\ge 2$ is a given integer. 
Then, for each string $\bm{s}=s_1\cdots s_k$ with digits $s_j \in \{0,1,\ldots,b-1\}$ and each $n\ge 1$, write $\pi_{b,\bm{s},n}(x)$ for the proportion of strings $\bm{s}$ in the $b$-adic expansion of $x$ which start at some position $\le n$, i.e.,
$$
\pi_{b,\bm{s},n}(x):=
\frac{\#\{i\in \{1,\ldots,n\}: d_{b,i+j-1}(x)=s_j \text{ for all }j=1,\ldots,k\}}{n}.
$$
In addition, let $S_{b}^k$ be the set of all possible strings $\bm{s}=s_1\cdots s_k$ in base $b$ of length $k$, hence $\#S_{b}^k=b^k$, and denote by $\bm{\pi}^{k}_{b,n}(x)$ the vector $(\pi_{b,\bm{s},n}(x): \bm{s} \in S_{b}^k)$. Of course, $\bm{\pi}^{k}_{b,n}(x)$ belongs to the $(b^k-1)$-dimensional simplex for each $n$.   
However, the components of $\bm{\pi}^{k}_{b,n}(x)$ satisfy an additional requirement: if $k\ge 2$ and $\bm{s}=s_1\cdots s_{k-1}$ is a string in $S_b^{k-1}$, then 
$$
\pi_{b,\bm{s},n}(x)
=\sum\nolimits_{s_k}\pi_{b,\bm{s} s_{k},n}(x)
=\sum\nolimits_{s_0}\pi_{b,s_0\bm{s},n}(x)+O\left(1/n\right)
 \quad \quad \text{ as }n\to \infty,
$$
where $s_0\bm{s}$ and $\bm{s}s_k$ stand for the concatened strings (indeed, the above identity is obtained by a double counting of the occurrences of the string $\bm{s}$ as the occurrences of all possible strings $\bm{s}s_k$; or, equivalently, as the occurrences of all possible strings $s_0\bm{s}$, with the caveat of counting them correctly at the two extreme positions, hence with an error of at most $1$). It follows that the set $\mathrm{L}^k_{b}(x)$ of accumulation points of the sequence of vectors $(\bm{\pi}^{k}_{b,n}(x): n\ge 1)$ is contained in $\Delta_{b}^k$, where
\begin{displaymath}
\begin{split}
\Delta_{b}^k:=\left\{(p_{\bm{s}})_{\bm{s} \in S_{b}^k} \in \mathbf{R}^{b^k}:\sum\nolimits_{\bm{s}} p_{\bm{s}}=1, \right. &
p_{\bm{s}}\ge 0 \text{ for all }\bm{s} \in S_{b}^k,  \\
&\left. 
\text{ and }\sum\nolimits_{s_0}p_{s_0\bm{s}}=\sum\nolimits_{s_k}p_{\bm{s}s_k} \text{ for all }\bm{s} \in S_{b}^{k-1}\right\}.
\end{split}
\end{displaymath}

Then $x$ is said to be \emph{normal} if 
$$
\forall b\ge 2, \forall k\ge 1, \forall \bm{s} \in S_{b}^k, \quad 
\lim_{n\to \infty} \pi_{b,\bm{s},n}(x)=1/b^{k}.
$$ 
Hence, if $x$ is normal, then $\mathrm{L}_{b}^k(x)=\{(1/b^{k}, \ldots, 1/b^{k})\}$. 
Olsen proved in \cite{MR2075041} that the subset of nonnormal numbers with maximal set of accumulation points 
is topologically large: 
\begin{thm}\label{thm:olson2004}
The set $\{x \in (0,1]: \mathrm{L}_{b}^k(x)=\Delta_{b}^k \text{ for all }b\ge 2, k\ge 1\}$ is comeager. 
\end{thm}

First, we strenghten Theorem \ref{thm:olson2004} by showing that the set of accumulation points $\mathrm{L}_{b}^k(x)$ can be replaced by the much smaller subset of 
accumulation points $\bm{\eta}$ such that every neighborhood of $\bm{\eta}$ contains \textquotedblleft sufficiently many\textquotedblright\,elements of the sequence, where \textquotedblleft sufficiently many\textquotedblright\,is meant with respect to a suitable ideal $\mathcal{I}$ of subsets of the positive integers $\mathbf{N}$; see Theorem \ref{thm:main}. 
Hence, Theorem \ref{thm:olson2004} corresponds to the case where $\mathcal{I}$ is the family of finite sets. 
 
Then, for certain ideals $\mathcal{I}$ (including the case of the family of asymptotic density zero sets), we even strenghten the latter result by showing that each accumulation point $\bm{\eta}$ can be chosen to be the limit of a subsequence with \textquotedblleft sufficiently many\textquotedblright\,indexes (as we will see in the next Section, these additional requirements are not equivalent); see Theorem \ref{thm:mainlimit}. 
The precise definitions, together with the main results, follow in Section \ref{sec:main}.


\section{Main results}\label{sec:main}

An ideal $\mathcal{I}\subseteq \mathcal{P}(\mathbf{N})$ is a family closed under finite union and subsets. 
It is also assumed that $\mathcal{I}$ contains the family of finite sets $\mathrm{Fin}$ and it is different from $\mathcal{P}(\mathbf{N})$. 
Every subset of $\mathcal{P}(\mathbf{N})$ is endowed with the relative Cantor-space topology. 
In particular, we may speak about $G_\delta$-subsets of $\mathcal{P}(\mathbf{N})$, $F_\sigma$-ideals, meager ideals, analytic ideals, etc. 
In addition, we say that $\mathcal{I}$ is a P-ideal if it is $\sigma$-directed modulo finite sets, i.e., for each sequence $(S_n)$ of sets in $\mathcal{I}$ there exists $S \in \mathcal{I}$ such that $S_n\setminus S$ is finite for all $n \in \mathbf{N}$.  
Lastly, we denote by $\mathcal{Z}$ the ideal of asymptotic density zero sets, i.e., 
\begin{equation}\label{eq:definitionZ}
\mathcal{Z}=\left\{S\subseteq \mathbf{N}: \mathsf{d}^\star(S)=0\right\},
\end{equation}
where $\mathsf{d}^\star(S):=\limsup_n \frac{1}{n}\#(S\cap [1,n])$ stands for the upper asymptotic density of $S$, see e.g. \cite{MR4054777}. 
We refer to \cite{MR2777744} for a recent survey on ideals and associated filters.

Let $x=(x_n)$ be a sequence taking values in a topological vector space $X$.  Then we say that $\eta \in X$ is an $\mathcal{I}$\emph{-cluster point} of $x$ if $\{n \in \mathbf{N}: x_n \in U\} \notin \mathcal{I}$ for all open neighborhoods $U$ of $\eta$. 
Note that $\mathrm{Fin}$-cluster points are the ordinary accumulation points. 
Usually $\mathcal{Z}$-cluster points are referred to as \emph{statistical cluster points}, see e.g. \cite{MR1181163}. 
It is worth noting that $\mathcal{I}$-cluster points have
been studied much before under a different name. 
Indeed, as it follows by \cite[Theorem 4.2]{MR3920799} and \cite[Lemma 2.2]{MR4126774}, they correspond to classical “cluster points” of a filter (depending on $x$) on the underlying space, cf. \cite[Definition 2, p.69]{MR1726779}. 

With these premises, for each $x \in (0,1]$ and for all integers $b\ge 2$ and $k\ge 1$, let $\Gamma_b^k (x,\mathcal{I})$ be the set of $\mathcal{I}$-cluster points of the sequence $(\bm{\pi}_{b,n}^k(x):n\ge 1)$.  
\begin{thm}\label{thm:main}
The set $\{x \in (0,1]: \Gamma_b^k (x,\mathcal{I})=\Delta_{b}^k \text{ for all }b\ge 2,k\ge 1\}$ is comeager, provided that $\mathcal{I}$ is a meager ideal. 
\end{thm}

The class of meager ideals is really broad. Indeed, it contains $\mathrm{Fin}$, $\mathcal{Z}$, the summable ideal $\{S\subseteq \mathbf{N}: \sum_{n \in S}1/n<\infty\}$, the ideal generated by the upper Banach density, the analytic P-ideals, the Fubini sum $\mathrm{Fin}\times \mathrm{Fin}$, the random graph ideal, etc.; cf. e.g. \cite{MPS, MR2777744}. 
Note 
that $\Gamma_b^k (x,\mathcal{I})=\mathrm{L}_b^k(x)$ if $\mathcal{I}=\mathrm{Fin}$. 
Therefore Theorem \ref{thm:main} significantly strenghtens Theorem \ref{thm:olson2004}.  
\begin{rmk}\label{rmk:maximal}
It is not difficult to see that Theorem \ref{thm:main} does not hold without any restriction on $\mathcal{I}$. Indeed, if $\mathcal{I}$ is a maximal ideal (i.e., the complement of a free ultrafilter on $\mathbf{N})$, then for each $x \in (0,1]$ and all integers $b\ge 2$, $k\ge 1$, we have that the sequence $(\bm{\pi}_{b,n}^k(x):n\ge 1)$ is bounded, hence it is $\mathcal{I}$-convergent so that $\Gamma_b^k (x,\mathcal{I})$ is a singleton. 
\end{rmk}

On a similar direction, if $x=(x_n)$ is a sequence taking values in a topological vector space $X$, then $\eta \in X$ is an $\mathcal{I}$\emph{-limit point} of $x$ if there exists a subsequence $(x_{n_k})$ such that $\lim_k x_{n_k}=\eta$ and $\mathbf{N}\setminus \{n_1,n_2,\ldots\} \in \mathcal{I}$. 
Usually $\mathcal{Z}$-limit points are referred to as \emph{statistical limit points}, see e.g. \cite{MR1181163}. 
Similarly, for each $x \in (0,1]$ and for all integers $b\ge 2$ and $k\ge 1$, let $\Lambda_b^k (x,\mathcal{I})$ be the set of $\mathcal{I}$-limit points of the sequence $(\bm{\pi}_{b,n}^k(x):n\ge 1)$. The analogue of Theorem \ref{thm:main} for $\mathcal{I}$-limit points follows. 
\begin{thm}\label{thm:mainlimit}
The set $\{x \in (0,1]: \Lambda_b^k (x,\mathcal{I})=\Delta_{b}^k \text{ for all }b\ge 2,k\ge 1\}$ is comeager, provided that $\mathcal{I}$ is an analytic P-ideal or an $F_\sigma$-ideal. 
\end{thm}

It is known that every $\mathcal{I}$-limit point is always an $\mathcal{I}$-cluster point, however they can be highly different, as it is shown in \cite[Theorem 3.1]{MR3883171}. 
This implies that Theorem \ref{thm:mainlimit} provides a further improvement on Theorem \ref{thm:main} for the subfamily of analytic P-ideals. 

It is remarkable that there exist $F_\sigma$-ideals which are not P-ideals, see e.g. \cite[Section 1.11]{MR1711328}. 
Also, the family of analytic P-ideals is well understood and has been characterized with the aid of lower semicontinuous submeasures, cf. Section \ref{sec:proofs}. 
The results in \cite{MR4124855} suggest that the study of the interplay between the theory of analytic P-ideals and their representability may have some relevant yet unexploited potential for the study of the geometry of Banach spaces.

Finally, recalling that the ideal $\mathcal{Z}$ defined in \eqref{eq:definitionZ} is an analytic P-ideal, an immediate consequence of Theorem \ref{thm:mainlimit} (as pointed out in the abstract) follows:
\begin{cor}\label{cor:statconv}
The set of $x \in (0,1]$ such that, for all $b\ge 2$ and $k\ge 1$, every vector in $\Delta_b^k$ is a statistical limit point of the sequence $(\bm{\pi}_{b,n}^k(x): n\ge 1)$ is comeager. 
\end{cor}

It would also be interesting to investigate to what extend the same results for nonnormal points belonging to self-similar fractals (as studied, e.g., by Olsen and West in \cite{MR4142282} in the context of iterated function systems) are valid. 

We leave as open question for the interested reader to check whether Theorem \ref{thm:mainlimit} can be extended for all $F_{\sigma\delta}$-ideals including, in particular, the ideal $\mathcal{I}$ generated by the upper Banach density (which is known to not be a P-ideal, see e.g. \cite[p.299]{MR632187}). 

\section{Proofs of the main results}\label{sec:proofs}

\begin{proof}[Proof of Theorem \ref{thm:main}]
Let $\mathcal{I}$ be a meager ideal on $\mathbf{N}$. It follows by Talagrand's characterization of meager ideals \cite[Theorem 21]{MR579439} that it is possible to define a partition $\{I_1,I_2,\ldots\}$ of $\mathbf{N}$ into nonempty finite subsets such that $S\notin \mathcal{I}$ whenever $I_n\subseteq S$ for infinitely many $n$. Moreover, we can assume without loss of generality that $\max I_n<\min I_{n+1}$ for all $n \in \mathbf{N}$. 

The claimed set can be rewritten as $\bigcap_{b\ge 2}\bigcap_{k\ge 1}X_b^k$, where $X_{b}^k:=\{x \in (0,1]: \Gamma_b^k (x,\mathcal{I})=\Delta_{b}^k\}$. 
Since the family of meager subsets of $(0,1]$ is a $\sigma$-ideal, it is enough to show that the complement 
of each $X_{b}^k$ is meager. 
To this aim, fix $b\ge 2$ and $k\ge 1$ and denote by $\|\cdot \|$ the Euclidean norm on $\mathbf{R}^{b^k}$. 
Considering that $\{\bm{\eta}_1, \bm{\eta}_2, \ldots\}:=\Delta_b^k \cap \mathbf{Q}^{b^k}$ is a countable dense subset of $\Delta_b^k$ and that $\Gamma_b^k(x,\mathcal{I})$ is a closed subset of $\Delta_b^k$ by \cite[Lemma 3.1(iv)]{MR3920799}, it follows that 
\begin{displaymath}
\begin{split}
(0,1]\setminus X_b^k
&= \bigcup\nolimits_{t\ge 1}\{x \in (0,1]: \bm{\eta}_t \notin \Gamma_b^k(x, \mathcal{I})\}\\
&= \bigcup\nolimits_{t\ge 1}\{x \in (0,1]: \exists \varepsilon>0, \{n \in \mathbf{N}: \|\bm{\pi}_{b,n}^k(x)-\bm{\eta}_t\|< \varepsilon\}\in \mathcal{I}\}\\
&\subseteq \bigcup\nolimits_{t, p, m\ge 1}
\{x \in (0,1]: \forall q\ge p, \exists n \in I_q, \|\bm{\pi}_{b,n}^k(x)-\bm{\eta}_t\|\ge\nicefrac{1}{m} \}.\\
\end{split}
\end{displaymath}
Denote by $S_{t,m,p}$ the set in the latter union. 
Thus 
it is sufficient to show that each $S_{t,p,m}$ is nowhere dense.  
To this aim, fix $t,p,m \in \mathbf{N}$ and a nonempty relatively open set $G\subseteq (0,1]$. We claim there exists a nonempty open set $U$ contained in $G$ and disjoint from $S_{t,p,m}$. Since $G$ is nonempty and open in $(0,1]$, there exists a string $\tilde{\bm{s}}=s_1\cdots s_j \in S_b^j$ such that $x \in G$ whenever $d_{b,i}(x)=s_i$ for all $i=1,\ldots,j$. Now, pick $x^\star \in (0,1]$ such that $\lim_n \bm{\pi}_{b,n}^k(x^\star)=\bm{\eta}_t$, which exists by \cite[Theorem 1]{MR2034019}. In addition, we can assume without loss of generality that $d_{b,i}(x^\star)=s_i$ for all $i=1,\ldots,j$. Since $\bm{\pi}_{b,n}^k(x^\star)$ is convergent to $\bm{\eta}_t$, there exists $q \ge p+j$ such that $\|\bm{\pi}^k_{b,n}(x^\star)-\bm{\eta}_{t} \|< \nicefrac{1}{m}$ for all $n \ge \min I_{q}$. 
Define $V:=\{x \in (0,1]: d_{b,i}(x)=d_{b,i}(x^\star) \text{ for all }i=1,\ldots,\max I_{q}+k\}$ and note that $V\subseteq G$ because $d_{b,i}(x)=s_i$ for all $i\le j$ and $x \in V$, and $V\cap S_{t,m,p}=\emptyset$ because, for each $x\in V$, the required property is not satisfied for this choice of $q$ since $\bm{\pi}_{b,n}^k(x)=\bm{\pi}_{b,n}^k(x^\star)$ for all $n \le \max I_q$. 
Clearly, $V$ has nonempty interior, hence it is possible to choose such $U\subseteq V$. 

This proves that each $S_{t,m,p}$ is nowhere dense, concluding the proof. 
\end{proof}

Before we proceed to the proof of Theorem \ref{thm:mainlimit}, we need to recall the classical Solecki's characterization of analytic P-ideals.  
A lower semicontinuous submeasure (in short, lscsm) is a monotone subadditive function $\varphi: \mathcal{P}(\mathbf{N}) \to [0,\infty]$ such that $\varphi(\emptyset)=0$, $\varphi(\{n\})<\infty$, and $\varphi(A)=\lim_m \varphi(A\cap [1,m])$ for all $A\subseteq \mathbf{N}$ and $n \in \mathbf{N}$. 
It follows by \cite[Theorem 3.1]{MR1708146} that an ideal $\mathcal{I}$ is an analytic P-ideal if and only if there exists a lscsm $\varphi$ such that
\begin{equation}\label{eq:analyticPideal}
\mathcal{I}=\{A\subseteq \mathbf{N}: \|A\|_\varphi=0\},
\,\,\, \|\mathbf{N}\|_\varphi =1,\,\,\,\text{ and }\,\,\varphi(\mathbf{N})<\infty.
\end{equation}
Here, $\|A\|_\varphi:=\lim_n \varphi(A\setminus [1,n])$ for all $A\subseteq \mathbf{N}$. Note that $\|A\|_\varphi=\|B\|_\varphi$ whenever the symmetric difference $A\bigtriangleup B$ is finite, cf. \cite[Lemma 1.3.3(b)]{MR1711328}. 
Easy examples of lscsms are $\varphi(A):=\# A$ or $\varphi(A):=\sup_n \frac{1}{n}\#(A \cap [1,n])$ for all $A\subseteq \mathbf{N}$ which lead, respectively, to the ideals $\mathrm{Fin}$ and $\mathcal{Z}$ through the representation \eqref{eq:analyticPideal}.  

\begin{proof}[Proof of Theorem \ref{thm:mainlimit}]
First, let us suppose that $\mathcal{I}$ is an $F_\sigma$-ideal. We obtain by \cite[Theorem 2.3]{MR3883171} that $\Lambda_b^k (x,\mathcal{I})=\Gamma_b^k (x,\mathcal{I})$ for each $b\ge 2$, $k\ge 1$, and $x \in (0,1]$. Therefore the claim follows by Theorem \ref{thm:main}. 

Then, we assume hereafter that $\mathcal{I}$ is an analytic P-ideal generated by a lscsm $\varphi$ as in \eqref{eq:analyticPideal}. Fix integers $b\ge 2$ and $k\ge 1$, and define the function
$$
\mathfrak{u}: (0,1]\times \Delta_b^k \to \mathbf{R}: (x,\bm{\eta}) \mapsto \lim_{t\to \infty} \|\{n \in \mathbf{N}: \|\bm{\pi}_{b,n}^k(x)-\bm{\eta}\|\le \nicefrac{1}{t} \}\|_\varphi 
$$
where $\|\cdot\|$ stands for the Euclidean norm on $\mathbf{R}^{b^k}$. It follows by \cite[Lemma 2.1]{MR3883171} that every section $\mathfrak{u}(x,\cdot)$ is upper semicontinuous, so that the set 
$$
\Lambda_b^k(x,\mathcal{I},q):=\{\bm{\eta} \in \Delta_b^k: \mathfrak{u}(x,\bm{\eta}) \ge q\}
$$
is closed for each $x \in (0,1]$ and $q \in \mathbf{R}$.

At this point, we prove that, for each $\bm{\eta} \in \Delta_b^k$, the set 
$
X(\bm{\eta}):=\{x \in (0,1]: \mathfrak{u}(x,\bm{\eta}) \ge \nicefrac{1}{2}\}
$ 
is comeager. To this aim, fix $\bm{\eta} \in \Delta_b^k$ and notice that 
\begin{displaymath}
\begin{split}
(0,1]\setminus X(\bm{\eta})
&=\bigcup\nolimits_{t\ge 1}\{x \in (0,1]: \|\{n \in \mathbf{N}: \|\bm{\pi}_{b,n}^k(x)-\bm{\eta}\|\le \nicefrac{1}{t}\}\|_\varphi < \nicefrac{1}{2}\}\\
&=\bigcup\nolimits_{t\ge 1}\{x \in (0,1]:  \lim_{h\to \infty} \varphi(\{n\ge h: \|\bm{\pi}_{b,n}^k(x)-\bm{\eta}\|\le \nicefrac{1}{t}\})<\nicefrac{1}{2}\}\\
&=\bigcup\nolimits_{t,h\ge 1}\{x \in (0,1]: \varphi(\{n\ge h: \|\bm{\pi}_{b,n}^k(x)-\bm{\eta}\|\le \nicefrac{1}{t}\})<\nicefrac{1}{2}\}.
\end{split}
\end{displaymath}
Denoting by $Y_{t,h}$ the inner set above, it is sufficient to show that each $Y_{t,h}$ is nowhere dense. 
Hence, fix $G\subseteq (0,1]$, $\tilde{\bm{s}} \in S_b^j$, and $x^\star \in (0,1]$ as in the proof of Theorem \ref{thm:main}. 
Considering that $\|\cdot\|_\varphi$ is invariant under finite sets, it follows that
\begin{displaymath}
\varphi(\{n\ge j^\prime: \|\bm{\pi}_{b,n}^k(x^\star)-\bm{\eta}\|\le \nicefrac{1}{t}\})
\ge \|\{n\ge j^\prime: \|\bm{\pi}_{b,n}^k(x^\star)-\bm{\eta}\|\le \nicefrac{1}{t}\}\|_\varphi=\mathfrak{u}(x^\star, \bm{\eta})=1, 
\end{displaymath}
where $j^\prime:=j+h$. Since $\varphi$ is lower semicontinuous, there exists an integer $j^{\prime\prime}>j^\prime$ such that 
$$
\varphi(\{n\in [j^\prime, j^{\prime\prime}]: \|\bm{\pi}_{b,n}^k(x^\star)-\bm{\eta}\|\le \nicefrac{1}{t}\})\ge \nicefrac{1}{2}.
$$ 
Define $V:=\{x \in (0,1]: d_{b,i}(x)=d_{b,i}(x^\star) \text{ for all }i=1,\ldots,j^{\prime\prime}\}$. Similarly, note that $V\subseteq G$ because $d_{b,i}(x)=s_i$ for all $i\le j$ and $x \in V$, and $V \cap Y_{t,h}=\emptyset$ because $\varphi(\{n\ge h: \|\bm{\pi}_{b,n}^k(x)-\bm{\eta}\|\le \nicefrac{1}{t}\})$ is at least $\varphi(\{n\in [j^\prime, j^{\prime\prime}]: \|\bm{\pi}_{b,n}^k(x)-\bm{\eta}\|\le \nicefrac{1}{t}\})\ge \nicefrac{1}{2}$ for all $x \in V$. 
Since $V$ has nonempty interior, it is possible to choose $U\subseteq V$ with the required property.

Finally, let $E$ be a countable dense subset of $\Delta_b^k$. 
Considering that $X:=\{x \in (0,1]: E\subseteq \Lambda_b^k(x,\mathcal{I},\nicefrac{1}{2})\}$ is equal to $\bigcap_{\bm{\eta} \in E}X(\bm{\eta})$, it follows that the set 
$
X
$ 
is comeager. However, considering that 
$$
\Lambda_b^k(x,\mathcal{I})=\bigcup\nolimits_{q>0}\Lambda_b^k(x,\mathcal{I},q)
$$
by \cite[Theorem 2.2]{MR3883171} and that $\Lambda_b^k(x,\mathcal{I},\nicefrac{1}{2})$ is a closed subset such that $E\subseteq \Lambda_b^k(x,\mathcal{I},\nicefrac{1}{2})\subseteq \Lambda_b^k(x,\mathcal{I})\subseteq \Delta_b^k$ for all $x \in X$, we obtain that $\Lambda_b^k(x,\mathcal{I}, \nicefrac{1}{2})=\Lambda_b^k(x,\mathcal{I})=\Delta_b^k$ for all $x \in X$. In particular, the claimed set contains $X$, which is comeager. This concludes the proof. 
\end{proof}

\section{Applications}

\subsection{Hausdorff and packing dimensions} 
We refer to \cite[Chapter 3]{MR3236784} for the definitions of the Hausdorff dimension and the packing dimension.
\begin{prop}\label{prop:hausdpack} 
The sets defined in Theorem \ref{thm:main} and Theorem \ref{thm:mainlimit} have Hausdorff dimension $0$ and packing dimension $1$.
\end{prop}
\begin{proof}
Reasoning as in \cite{MR2075041}, the claimed sets are contained in the corresponding ones with ideal $\mathrm{Fin}$, which have Hausdorff dimension $0$ by \cite[Theorem 2.1]{MR2034019}. In addition, since all sets are comeager, we conclude that they have packing dimension $1$ by  \cite[Corollary 3.10(b)]{MR3236784}. 
\end{proof}

\subsection{Regular matrices} 
We extend the main results contained in \cite{MR2672291, Stylianou}. To this aim, let $A=(a_{n,i}: n,i \in \mathbf{N})$ be a \emph{regular matrix}, that is, an infinite real-valued matrix such that, if $\bm{z}=(\bm{z}_n)$ is a $\mathbf{R}^d$-valued sequence convergent to $\bm{\eta}$, then $A_n\bm{z}:=\sum_i a_{n,i}\bm{z}_i$ exists for all $n \in \mathbf{N}$ and 
$\lim_n A_n\bm{z}=\bm{\eta}$, 
see e.g. 
\cite[Chapter 4]{MR0193472}. 
Then, for each $x \in (0,1]$ and integers $b\ge 2$ and $k\ge 1$, let $\Gamma_{b}^k(x,\mathcal{I},A)$ be the set of $\mathcal{I}$-cluster points of the sequence of vectors $\left(A_n\bm{\pi}_{b}^k(x): n \ge 1\right)$, 
where $\bm{\pi}_{b}^k(x)$ is the sequence $(\bm{\pi}_{b,n}^k(x): n\ge 1)$. 

In particular, $\Gamma_{b}^k(x,\mathcal{I},A)=\Gamma_{b}^k(x,\mathcal{I})$ if $A$ is the infinite identity matrix. 

\begin{thm}\label{cor:regular}
The set 
$\{x \in (0,1]: \Gamma_b^k (x,\mathcal{I},A)\supseteq \Delta_{b}^k \text{ for all }b\ge 2, k\ge 1\}$ is comeager, provided that $\mathcal{I}$ is a meager ideal and $A$ is a regular matrix. 
\end{thm}
\begin{proof} 
Fix a regular matrix $A=(a_{n,i})$ and a meager ideal $\mathcal{I}$. The proof goes along the same lines as the proof of Theorem \ref{thm:main},  
replacing the definition of $S_{t,m,p}$ with 
$$
S^\prime_{t,m,p}:=
\{x \in (0,1]: 
\forall q\ge p, \exists n \in I_q, 
\|A_n\bm{\pi}^k_{b}(x)-\bm{\eta}_{t} \|\ge \nicefrac{1}{m} \}. 
$$
Recall that, thanks to the classical Silverman--Toeplitz characterization of regular matrices, see e.g. 
\cite[Theorem 4.1, II]{MR0193472} or \cite{LeoCon21}, 
we have that $\sup_n \sum_i |a_{n,i}|<\infty$. 
Since $\lim_n \bm{\pi}_{b,n}^k(x^\star)=\bm{\eta}_t$, it follows that there exist sufficiently large integers $q\ge p+j$ and $j_A \ge j$ such that, if $d_{b,i}(x)=d_{b,i}(x^\star)$ for all $i=1,\ldots,j_A+k$, then 
\begin{equation}\label{eq:inequalityregularanalogue}
\begin{split}
\|A_n\bm{\pi}^k_{b}(x)-\bm{\eta}_{t} \|
&\le \|A_n\bm{\pi}^k_{b, }(x^\star)-\bm{\eta}_{t}\|+\left\|\sum\nolimits_i a_{n,i}(\bm{\pi}^k_{b,i}(x)-\bm{\pi}^k_{b,i}(x^\star))\right\|\\
&\le \|A_n\bm{\pi}^k_{b }(x^\star)-\bm{\eta}_{t} \|+\sum\nolimits_i |a_{n,i}| \, \|\bm{\pi}^k_{b,i}(x)-\bm{\pi}^k_{b,i}(x^\star)\|\\
&\le \|A_n\bm{\pi}^k_{b}(x^\star)-\bm{\eta}_{t}\|+\sum\nolimits_{i> j_A} |a_{n,i}| 
< \frac{1}{m}
\end{split}
\end{equation}
for all $n \in I_{q}$. We conclude analogously that $S^\prime_{t,m,p}$ is nowhere dense. 
\end{proof}

The main result in \cite{Stylianou} corresponds to the case $\mathcal{I}=\mathrm{Fin}$ and $k=1$, although with a different proof; cf. also Example \ref{example10} below.  

At this point, we need an intermediate result which is of independent interest. For each bounded sequence $\bm{x}=(\bm{x}_n)$ with values in $\mathbf{R}^k$, let $\mathrm{K}\text{-}\mathrm{core}(\bm{x})$ be the \emph{Knopp core} of $\bm{x}$, that is, the convex hull of the set of accumulation points of $\bm{x}$. In other words, $\mathrm{K}\text{-}\mathrm{core}(\bm{x})=\mathrm{co}\, \mathrm{L}_{\bm{x}}$, where $\mathrm{co}\, S$ is the convex hull of  $S\subseteq \mathbf{R}^k$ and $\mathrm{L}_{\bm{x}}$ is the set of accumulation points of $\bm{x}$. 
The ideal version of the Knopp core has been studied in \cite{LeoJMAA20, MR4126774}.
The classical Knopp theorem states that, if $k=2$ and $A$ is a nonnegative regular matrix, then 
\begin{equation}\label{eq:inclusioncoresFin}
\mathrm{K}\text{-}\mathrm{core}(A\bm{x})\subseteq \mathrm{K}\text{-}\mathrm{core}(\bm{x})
\end{equation}
for all bounded sequences $\bm{x}$, where $A\bm{x}=(A_n\bm{x}: n\ge 1)$, see \cite[p. 115]{MR1545100}; cf. \cite[Chapter 6]{MR0193472} for a textbook exposition. A generalization in the case $k=1$ can be found in \cite{MR549529}. 
We show, in particular, that a stronger version of Knopp's theorem holds for every $k \in \mathbf{N}$. 
\begin{prop}\label{prop:inclusioncores}
Let $\bm{x}=(\bm{x}_n)$ be a bounded sequence taking values in $\mathbf{R}^k$, and fix a regular matrix $A$ such that $\lim_n \sum_i |a_{n,i}|=1$. Then inclusion \eqref{eq:inclusioncoresFin} holds. 
\end{prop}
\begin{proof} 
Define $\kappa:=\sup_n \|\bm{x}_n\|$ and let $\bm{\eta}$ be an accumulation point of $A\bm{x}$. 
It is sufficient to show that $\bm{\eta}\in K:=\mathrm{K}\text{-}\mathrm{core}(\bm{x})$. Possibly deleting some rows of $A$, we can assume without loss of generality that $\lim A\bm{x}=\bm{\eta}$. For each $m \in \mathbf{N}$, let $K_m$ be the closure of $\mathrm{co}\{x_m,x_{m+1},\ldots\}$, hence $K\subseteq K_m$. 
Define $d(\bm{a}, C):=\min_{\bm{b} \in C} \|\bm{a}-\bm{b}\|$ for all $\bm{a} \in \mathbf{R}^k$ and nonempty compact sets $C\subseteq \mathbf{R}^k$. 
In addition, for each $m \in \mathbf{N}$, let $Q_m(\bm{a})\in K_m$ be the unique vector such that $d(\bm{a}, K_m)=\|\bm{a}-Q_m(\bm{a})\|$. Similarly, let $Q(\bm{a})$ be the vector in $K$ which minimizes its distance with $\bm{a}$. 
Then, notice that, for all $n,m \in \mathbf{N}$, we have 
\begin{displaymath}
\begin{split}
d(A_n\bm{x}, K) &
\le \inf\nolimits_{\bm{b} \in K} \inf\nolimits_{\bm{c} \in \mathbf{R}^k} (\|A_n\bm{x}-\bm{c}\|+\|\bm{c}-\bm{b}\|) \\
&\le \inf\nolimits_{\bm{c} \in K_m} \inf\nolimits_{\bm{b} \in K}  (\|A_n\bm{x}-\bm{c}\|+\|\bm{c}-\bm{b}\|) \\
&\le \inf\nolimits_{\bm{c} \in K_m} \|A_n\bm{x}-\bm{c}\|+\sup\nolimits_{\bm{y}\in K_m} \inf\nolimits_{\bm{b} \in K}\|\bm{y}-\bm{b}\| \\
&= d(A_n\bm{x},K_m)+\sup\nolimits_{\bm{y}\in K_m}d(\bm{y},K)
\end{split}
\end{displaymath}
Since $d(\bm{\eta},K)=\lim_n d(A_n\bm{x}, K)$ 
by the continuity of $d(\cdot,K)$, 
it is sufficient to show that both $d(A_n\bm{x}, K_m)$ and $\sup_{\bm{y} \in K_m} d(\bm{y},K)$ are sufficiently small if $n$ is sufficiently large and $m$ is chosen properly. 

To this aim, fix $\varepsilon >0$ and choose $m \in \mathbf{N}$ such that $\sup\nolimits_{\bm{y}\in K_m}d(\bm{y},K)\le \nicefrac{\varepsilon}{2}$. 
Indeed, it is sufficient to choose $m \in \mathbf{N}$ such that $d(\bm{x}_n, \mathrm{L}_{\bm{x}})< \nicefrac{\varepsilon}{2}$ for all $n\ge m$: 
indeed, in the opposite, the subsequence $(\bm{x}_j)_{j \in J}$, where $J:=\{n\in \mathbf{N}: d(\bm{x}_n, \mathrm{L}_{\bm{x}})\ge  \nicefrac{\varepsilon}{2}\}$, would be bounded and without any accumulation point, which is impossible. 
Now pick $\bm{y} \in K_m$ so that $\bm{y}=\sum_j \lambda_{i_j} \bm{x}_{i_j}$ for some strictly increasing sequence $(i_j)$ of positive integers such that $i_1 \ge m$ and some real nonnegative sequence $(\lambda_{i_j})$ with $\sum_{j} \lambda_{i_j}=1$. It follows that 
$$
d(\bm{y},K)\le \left\|\bm{y}-\sum\nolimits_{j}\lambda_{i_j}Q(\bm{x}_{i_j}) \right\|
\le \sum\nolimits_{j}\lambda_{i_j} \left\|\bm{x}_{i_j}-Q(\bm{x}_{i_j})\right\|
\le  \sum\nolimits_{j}\lambda_{i_j} d(\bm{x}_{i_j}, L_{\bm{x}}) 
\le \frac{\varepsilon}{2}.
$$

Suppose for the moment that $A$ has nonnegative entries. 
Since $A$ is regular, we get $\lim_n \sum_i a_{n,i}=1$ and $\lim_n \sum_{i<m}a_{n,i}=0$ by the Silverman--Toeplitz characterization, hence $\lim_n\sum_{i\ge m} a_{n,i}=1$ and there exists $n_0 \in \mathbf{N}$ such that $\sum_{i\ge m} a_{n,i} \ge \nicefrac{1}{2}$ for all $n\ge n_0$. Thus, for each $n\ge n_0$, we obtain that 
$
 d(A_n\bm{x},K_m)=\|A_n\bm{x}-Q_m(A_n\bm{x})\| \le \alpha_n+\beta_n+\gamma_n,
$ 
where 
$$
\alpha_n:=\left\| A_n\bm{x}-\frac{A_n\bm{x}}{\sum_i a_{n,i}}\right\|, \quad \beta_n:=\left\| \frac{A_n\bm{x}}{\sum_i a_{n,i}}-Q_m\left(\frac{A_n\bm{x}}{\sum_i a_{n,i}}\right)\right\|, 
$$
and 
$$
\gamma_n:=\left\| Q_m\left(\frac{A_n\bm{x}}{\sum_i a_{n,i}}\right)-Q_m(A_n\bm{x})\right\|.
$$
Recalling that $\kappa=\sup_n \|\bm{x}_n\|$, it is easy to see that
\begin{displaymath}
\gamma_n \le \alpha_n\le \kappa \sum\nolimits_i |a_{n,i}| \cdot \left(1-\frac{1}{\sum\nolimits_i a_{n,i}}\right).
\end{displaymath}
In addition, setting $t_n:=\sum_{i\ge m}a_{n,i}/\sum_i a_{n,i} \in [0,1]$ for all $n\ge n_0$, we get
\begin{equation}\label{eq:inequalitybeta}
\begin{split}
\beta_n
&\le \left\| \frac{\sum_i^\star}{\sum_i a_{n,i}}-\frac{\sum_{i\ge m}^\star}{\sum_{i\ge m} a_{n,i}}\right\|\\
&=\frac{1}{\sum\nolimits_{i\ge m}a_{n,i}\sum\nolimits_{i}a_{n,i}}\left\| \sum\nolimits_{i\ge m}a_{n,i}\left(\sum\nolimits_{i< m}^\star+\sum\nolimits_{i\ge m}^\star \right)
-\sum\nolimits_ia_{n,i}\sum\nolimits_{i\ge m}^\star\right\|\\
&=\frac{1}{\sum\nolimits_{i\ge m}a_{n,i}}\left\| t_n\sum\nolimits_{i< m}^\star +(1-t_n)\sum\nolimits_{i\ge m}^\star\right\|\\
&\le 2\kappa\left( t_n \sum\nolimits_{i<m} |a_{n,i}|+(1-t_n)\sum\nolimits_{i} |a_{n,i}|\right).
\end{split}
\end{equation}
where $\sum_{i \in I}^\star$ stands for $\sum_{i \in I}a_{n,i}\bm{x}_i$. 
Note that the hypothesis that the entries of $A$ are nonnegative has been used only in the first line of \eqref{eq:inequalitybeta}, so that $\sum^\star_{i\ge m}/\sum_{i\ge m}a_{n,i} \in K_m$. 
Since $\lim_n \sum_{i<m}|a_{n,i}|=0$, $\lim_n t_n=1$, and $\sup_n \sum_i |a_{n,i}|<\infty$ by the regularity of $A$, it follows that all $\alpha_n, \beta_n, \gamma_n$ are smaller than $\nicefrac{\varepsilon}{6}$ if $n$ is sufficiently large. 
Therefore $d(A_n\bm{x},K)\le \varepsilon$ and, since $\varepsilon$ is arbitrary, we conclude that $\bm{\eta}=\lim_n A_n\bm{x} \in K$. 

Lastly, suppose that $A$ is a regular matrix such that $\lim_n \sum_i |a_{n,i}|=1$ and let $B=(b_{n,i})$ be the nonnegative regular matrix defined by $b_{n,i}=|a_{n,i}|$ for all $n,i \in \mathbf{N}$. Considering that 
$$
d(A_n\bm{x}, K_m) \le \|A_n\bm{x}- B_n \bm{x}\|+d(B_n\bm{x}, K_m) \le \kappa \sum\nolimits_i |a_{n,i}-|a_{n,i}||+\varepsilon,
$$
and that $\lim_n\sum\nolimits_i |a_{n,i}-|a_{n,i}||= 0$ because $\lim_n \sum_i a_{n,i}=\lim_n\sum_i |a_{n,i}|=1$, we conclude that $d(A_n\bm{x}, K_m) \le 2\varepsilon$ whenever $n$ is sufficiently large. 
The claim follows as 
before. 
\end{proof} 
The following corollary is immediate:
\begin{cor}
Let $\bm{x}=(\bm{x}_n)$ be a bounded sequence taking values in $\mathbf{R}^k$, and fix a nonnegative regular matrix $A$. Then inclusion \eqref{eq:inclusioncoresFin} holds. 
\end{cor}
\begin{rmk}\label{rmk:maddoxnonnegative}
Inclusion \eqref{eq:inclusioncoresFin} fails for an arbitrary regular matrix: indeed, let $A=(a_{n,i})$ be the  matrix defined by $a_{n,2n}=2$, $a_{n,2n-1}=-1$ for all $n \in \mathbf{N}$, and $a_{n,i}=0$ otherwise. 
Set also $k=1$ and let $x$ be the sequence such that $x_n=(-1)^n$ for all $n\in \mathbf{N}$. Then $A$ is regular and $\lim Ax=3 \notin \{-1,1\}=\mathrm{K}\text{-}\mathrm{core}(x)$. 
\end{rmk}
\begin{rmk}
Proposition \ref{prop:inclusioncores} keeps holding on a (possibly infinite dimensional) Hilbert space $X$ with the following provisoes: replace the definition of $\mathrm{K}\text{-}\mathrm{core}(\bm{x})$ with the \emph{closure} of $\mathrm{co}\,\mathrm{L}_{\bm{x}}$ (this coincides in the case that $X=\mathbf{R}^k$) and assume that the sequence $\bm{x}$ is contained in a compact set (so that $\mathrm{K}\text{-}\mathrm{core}(\bm{x})$ is also nonempty). 
\end{rmk}

With these premises, 
we can strenghten Theorem \ref{cor:regular} as follows.
\begin{thm}\label{cor:nonnnegativeregular}
The set $\{x \in (0,1]: \Gamma_b^k (x,\mathcal{I},A)=\Delta_{b}^k \text{ for all }b\ge 2, k\ge 1\}$ is comeager, provided that $\mathcal{I}$ is a meager ideal and $A$ is a regular matrix such that $\lim_n\sum_i |a_{n,i}|=1$. 
\end{thm}
\begin{proof}
Let us suppose that $A=(a_{n,i})$ is nonnegative regular matrix, i.e., $a_{n,i}\ge 0$ for all $n,i \in \mathbf{N}$, and fix a meager ideal $\mathcal{I}$, a real $x \in (0,1]$, and integers $b\ge 2$, $k\ge 1$. 
Thanks to Theorem \ref{cor:regular}, 
it is sufficient to show that every accumulation point of the sequence $(A_n\bm{\pi}_{b}^k(x): n\ge 1)$ is contained in the convex hull of the set of accumulation points of $(\bm{\pi}_{b,n}^k(x): n\ge 1)$, which is in turn contained into $\Delta_b^k$. This follows by Proposition \ref{prop:inclusioncores}.
\end{proof}

Since the family of meager sets is a $\sigma$-ideal, the following is immediate by Theorem \ref{cor:nonnnegativeregular}.
\begin{cor}\label{cor:uniformnonnegative}
Let $\mathscr{A}$ be a countable family of regular matrices such that $\lim_n\sum_i |a_{n,i}|=1$. 
Then the set 
$\{x \in (0,1]: \Gamma_b^k (x,\mathcal{I},A)=\Delta_{b}^k \text{ for all }b\ge 2, k\ge 1, \text{ and all }A \in \mathscr{A}\}$ is comeager, provided that $\mathcal{I}$ is a meager ideal.
\end{cor}
It is worth to remark that the main result \cite{MR2672291} is obtained as an instance of Corollary \ref{cor:uniformnonnegative}, letting $\mathscr{A}$ be the set of iterates of the Cesàro matrix (note that they are nonnegative regular matrices), and setting $k=1$ and $\mathcal{I}=\mathrm{Fin}$. 
The same holds for the iterates of the H{\"o}lder matrix and the logarithmic Riesz matrix as in \cite[Sections 3 and 4]{MR4142282}.

Next, we show that the hypothesis $\lim_n \sum_i |a_{n,i}|=1$ for the entries of the regular matrix in Theorem  \ref{cor:nonnnegativeregular} cannot be removed. 
\begin{example}
Let $A=(a_{n,i})$ be the matrix such that $a_{n,(2n-1)!}=-1$ and $a_{n,(2n)!}=2$ for all $n \in \mathbf{N}$, and $a_{n,i}=0$ otherwise. It is easily seen that $A$ is regular. 
Then, set $b=2$, $k=1$, and $\mathcal{I}=\mathrm{Fin}$. 
We claim that the set of all $x \in (0,1]$ such that $2$ is an accumulation point of the sequence $\pi_{2,1}(x)=(\pi_{2,1,n}(x): n\ge 1)$ is comeager. Indeed, its complement can be rewritten as $\bigcup_{m,p}S_{m,p}$, where 
$$
S_{m,p}:=\{x \in (0,1]: |A_n\pi_{2,1}(x)-2|\ge \nicefrac{1}{m} \text{ for all }n\ge p\}.
$$ 
Let $x^\star\in (0,1]$ such that $d_{2,n}(x^\star)=1$ if and only if $(2i-1)!\le n<(2i)!$ for some $i \in \mathbf{N}$. Then it is easily seen that $\lim_n \pi_{2,1,n}(x^\star)=2$. 
Along the same lines of the proof of Theorem \ref{cor:regular}, it follows that each $S_{m,p}$ is meager. 
We conclude that $\{x \in (0,1]: \Gamma_2^1 (x,\mathrm{Fin},A)=\Delta_{2}^1\}$ is meager, which proves that the condition $\lim_n \sum_i |a_{n,i}|=1$ in the statement of Theorem \ref{cor:nonnnegativeregular} cannot be removed. 
\end{example}

In addition, the main result in \cite{Stylianou} states that Theorem \ref{cor:regular}, specialized to the case $\mathcal{I}=\mathrm{Fin}$ and $k=1$, can be further strengtened so that the set $\{x \in (0,1]: \Gamma_b^1 (x,\mathrm{Fin},A)\supseteq \Delta_{b}^1 \text{ for all }b\ge 2 \text{ and all regular } A\}$ is comeager. 
Taking into account the argument in the proof of Theorem \ref{cor:nonnnegativeregular}, this would imply that the set 
\begin{equation}\label{eq:setempty}
\{x \in (0,1]: \Gamma_b^1 (x,\mathrm{Fin},A)= \Delta_{b}^1 \text{ for all }b\ge 2 \text{ and all nonnegative regular } A\}
\end{equation}
should be comeager. 
However, this is false as it is shown in the next example.

\begin{example}\label{example10}
For each $y \in (0,1]$, let $(e_{y,k}: k\ge 1)$ be the increasing enumeration of the infinite set $\{n \in \mathbf{N}: d_{2,n}(y)=1\}$. 
Then, let $\mathscr{A}=\{A_y: y \in (0,1]\}$ be family of matrices $A_y=(a^{(y)}_{n,i})$ with entries in $\{0,1\}$ so that $a^{(y)}_{n,i}=1$ if and only if $e_{y,n}=i$ for all $y \in (0,1]$ and all $n,i \in \mathbf{N}$. Then each $A_y$ is a nonnegative regular matrix. 
It follows, for each ideal $\mathcal{I}$, 
$$
\{x \in (0,1]:  \Gamma_2^1 (x,\mathcal{I},A)=\Delta_{2}^1 
\text{ for all }
A \in \mathscr{A} \}=\emptyset.
$$
Indeed, for each $x \in (0,1]$, the sequence $\bm{\pi}_2^1(x)=(\bm{\pi}_{2,n}^1(x):n\ge 1)$ has an accumulation point $\bm{\eta} \in \Delta_2^1$. 
Hence there exists a subsequence $(\bm{\pi}_{2,n_k}^1(x):k\ge 1)$ which is convergent to $\bm{\eta}$. 
Equivalently, $\lim A_y\bm{\pi}_2^1(x)=\bm{\eta}$, where $y\in (0,1]$ is defined such that $e_{y,k}=n_k$ for all $k \in \mathbf{N}$. 
Therefore $\{\bm{\eta}\}=\Gamma_2^1 (x,\mathcal{I},A_y)\neq \Delta_{2}^1$. 
in particular, the set defined in \eqref{eq:setempty} is empty. 
\end{example}

Lastly, the analogues of Theorem \ref{cor:regular} and Theorem \ref{cor:nonnnegativeregular} hold for $\mathcal{I}$-limit points, if $\mathcal{I}$ is an $F_\sigma$-ideal or an analytic P-ideal. Indeed, denoting with $\Lambda_b^k(x,\mathcal{I},A)$ the set of $\mathcal{I}$-limit points of the sequence $(A_n\bm{\pi}_b^k(x): n\ge 1)$, we obtain:
\begin{thm}\label{thm:regularIlimitpoints}
Let $A$ be a regular matrix and let $\mathcal{I}$ be an $F_\sigma$-ideal or an analytic P-ideal. 
Then the set $\{x \in (0,1]: \Lambda_b^k (x,\mathcal{I},A)\supseteq \Delta_{b}^k \text{ for all }b\ge 2, k\ge 1\}$ is comeager.

Moreover, the set $\{x \in (0,1]: \Lambda_b^k (x,\mathcal{I},A)=\Delta_{b}^k \text{ for all }b\ge 2, k\ge 1\}$ is comeager if, in addition, $A$ satisfies $\lim_n\sum_i |a_{n,i}|=1$.
\end{thm}
\begin{proof}
The first part goes along the same lines of the proof of Theorem \ref{thm:mainlimit}. Here, we replace $\bm{\pi}_b^k(x)$ with $(A_n\bm{\pi}_b^k(x): n\ge 1)$ and using the chain of inequalities \eqref{eq:inequalityregularanalogue}: more precisely, we consider $j^{\prime\prime}\in \mathbf{N}$ such that 
$
\varphi(\{n \in [j^\prime, j^{\prime\prime}]: \|A_n \bm{\pi}^k_b(x^\prime)-\bm{\eta}\|\le \nicefrac{1}{2t}\})\ge \nicefrac{1}{2},
$
and, taking into considering \eqref{eq:inequalityregularanalogue}, we define $V:=\{x \in (0,1]: d_{b,i}(x)=d_{b,i}(x^\star) \text{ for all }i=1,\ldots,k+j^{\prime\prime\prime}\}$, where $j^{\prime\prime\prime}$ is a sufficiently large integer such that $\sum_{i>j^{\prime\prime\prime}}|a_{n,i}|\le \nicefrac{1}{2t}$ for all $n \in [j^\prime, j^{\prime\prime}]$. 

The second part follows, as in Theorem \ref{cor:nonnnegativeregular},  by the fact that every accumulation point of $(A_n\bm{\pi}_b^k(x): n\ge 1)$ belongs to $\Delta_b^k$.
\end{proof}

\section{Acknowledgments.} 
P.~Leonetti is grateful to PRIN 2017 (grant 2017CY2NCA) for financial support.

\bibliographystyle{amsplain}

\end{document}